\documentclass[12pt]{amsart}
\usepackage{amsmath}
\usepackage{amssymb}
\usepackage{bm}
\usepackage{esint}
\usepackage{graphicx}
\usepackage{verbatim}
\usepackage{enumitem}
\usepackage[normalem]{ulem}
\usepackage{stackengine}
\usepackage{algorithm}
\usepackage{algpseudocode}
\stackMath
\setlist[enumerate,1]{leftmargin=*}
\setlist[itemize,1]{leftmargin=*}
\usepackage{mathrsfs}
\usepackage[dvipsnames]{color}
\usepackage{todonotes}
\usepackage{float}
\usepackage{etoolbox}
\usepackage[pdftex,pdfstartview=FitH,pdfborderstyle={/S/B/W 1},%
colorlinks=true, linkcolor=blue, urlcolor=blue, citecolor=blue,%
pagebackref=true]{hyperref}
\usepackage[nobysame,alphabetic,initials,msc-links]{amsrefs}
\usepackage{imakeidx}
\usepackage{tikz}
\usetikzlibrary{decorations.markings}

\newcommand{\veps}{\varepsilon}
\newcommand{\vphi}{\varphi}

\newcommand{\NN}{\mathbb{N}}
\newcommand{\ZZ}{\mathbb{Z}}

\newcommand{\RR}{\mathbb{R}}

\newcommand{\set}[2]{\left\{#1 : #2\right\}}

\newcommand{\PP}{\mathbb{P}}
\newcommand{\EE}{\mathbb{E}}
\newcommand{\cF}{\mathcal{F}}
\newcommand{\cB}{\mathcal{B}}
\newcommand{\cP}{\mathcal{P}}

\newcommand{\cV}{\mathcal{V}}

\newtheorem{defn}{Definition}[section]
\newtheorem{thm}[defn]{Theorem}
\newtheorem{prop}[defn]{Proposition}

\newtheorem{cor}[defn]{Corollary}
\newtheorem{qn}[defn]{Question}
\newtheorem{rem}[defn]{Remark}

\numberwithin{equation}{section}

\begin{document}

\title{Sequence Entropy of Rank One Systems}
\author{Shigenori Takeda}
\address{S.T.: Department of Mathematics, National University of Singapore, 119076, Singapore}
\email{e0407799@u.nus.edu}
\date{\today}

\setlength{\marginparwidth}{2 cm}

\begin{abstract}
    We study the sequence entropy of rank one measure-preserving systems along subexponential sequences. We prove that the sequence entropy along a large class of sequences can be infinite using Ornstein's probabilistic constructions. Moreover, we show that sequence entropy necessarily vanishes for subexponential sequences if the growth of tower heights remains below certain growth rates, and obtain a flexibility result for polynomial sequences at this critical threshold. 
\end{abstract}

\maketitle

\section{Introduction}

The classification problem of measure-preserving systems has had a long history in the field of ergodic theory. A foundational concept is the Kolmogorov-Sinai entropy, a powerful measure-theoretic invariant that quantifies the complexity of a system by measuring the exponential rate of growth of distinguishable orbits. However, for systems with zero entropy, it offers no quantitative measure of complexity as it fails to detect subexponential system complexities.

More sensitive tools have since been developed to better understand the behaviour of zero entropy systems. One such quantity was sequence entropy, originally developed by Kushnirenko \cite{Kushnirenko67} to distinguish the horocyclic flow on two-dimensional hyperbolic surfaces from its Cartesian square. Since then, several fundamental properties of sequence entropy have been investigated, notably the relation between sequence entropy and mixing properties, see the recent survey on entropy-type invariants \cite{Kanigowski24}.

This article may be regarded as an addition to a series of works under the flexibility program of dynamical systems initiated by Anatole Katok \cite{BochiKatok22}. The general principle of flexibility suggests that dynamical invariants should take any arbitrary value, bar some well-understood obstruction innate to the invariant and the class of systems in consideration. Early results of this flavour primarily focused on classical invariants such as Lyapunov exponents of expanding maps on a circle \cite{Erchenko19} and Anosov diffeomorphisms \cite{BochiKatok22}. On the topic of Kolmogorov-Sinai entropy, flexibility results were obtained in geodesic flows \cite{ErchenkoKatok19}, piecewise expanding unimodal maps \cite{AlsedaMisiurewicz24} and Bowen-Series boundary maps \cite{AbramsKatok22}. A striking example appeared in \cite{KucherenkoQuas22} where an entire class of convex Lipschitz functions can be realised as topological pressure functions on continuous potentials, highlighting the versatility of the program.

Under this framework, a natural question about sequence entropy is as follows:

\begin{qn} \label{qn:seq_ent_flex}
    For which sequences $A=\{t_n\}_{n=1}^\infty$ do there exist a zero entropy system $(X,\cB,\mu,T)$ such that $h^A(T)$ takes any arbitrary value in $[0,+\infty]$?
\end{qn}

To investigate this problem, we focus on a general class of well-understood zero entropy systems that can serve as a rich source of examples and counterexamples. \emph{Rank one systems}, which made its way to the core of ergodic-theoretic constructions due to Ornstein \cite{Ornstein72}, perfectly fulfil this role. A measure-preserving system $(X,\cB,\mu,T)$ is said to be of rank one if there exists an increasing sequence of positive integers $h_n$ and a descending sequence of measurable subsets $F_n$ such that the partitions
\[\xi_n=\left\{F_n,TF_n,\cdots,T^{h_n-1}F_n,X \setminus \bigcup_{k=0}^{h_n-1}T^kF_n\right\}\]
form a sequence of refinements approaching the point partition $\veps$. The linear structure of a rank one system drastically simplifies the behaviour of $A$-orbits $\{T^{t_k}x\}$, yet still allows for highly complex behaviour over sufficiently sparse sequences of orbit lengths.
    
% Main Theorems

We will provide a semi-constructive proof that first answers in the positive the existence of zero entropy systems with infinite sequence entropy using rank one systems with tower heights $h_n$ admitting an unrestricted rate of growth. This result will apply to any sequence which \emph{dilates to infinity}, namely those that satisfy $\lim_{n \to \infty} (t_{n+1}-t_{n}) = +\infty$. The significance of enlarging gap sizes was already discovered in \cite{Ryzhikov21} where it was shown that for such sequences, infinite sequence entropy is in fact a generic phenomenon among automorphisms of a standard probability space. These results answer \cite{Kanigowski24}*{Question 6.4.3} in greater generality for sequence entropies along polynomial sequences (provided that they are non-linear, see the discussion leading up to Question~\ref{qn:Krug-Newton}).

\begin{thm} \label{seq_ent_blowup}
    Given any sequence of non-negative integers $A$ that dilates to infinity, there is a rank one system $(X,\cB,\mu,T)$ such that $h^A(T) = +\infty$.
\end{thm}

\begin{rem}
    A variant of sequence entropy, called $P$-entropy, was under consideration in \cite{Ryzhikov21}. The result can be applied directly to our case of Kushnirenko's sequence entropy by regrouping a sequence into portions with an expanding number of terms and widening gaps.
\end{rem}

\begin{rem} We would like to acknowledge that Dou and Park have also independently established an affirmative answer to \cite{Kanigowski24}*{Question 6.4.3} in their forthcoming work, utilising deterministic walks in random scenes.
\end{rem}

% \begin{rem}
%     In the proof of Theorem~\ref{seq_ent_blowup} we show that the following is true: for any integer $h \ge 2$, there is a rank one system $(X,\cB,\mu,T)$ such that the sequence of $(h_n+1)$-set partitions $\xi_{S_n}$ induced by its towers satisfies $h_1=h$ and
%     \[h^A(T,\xi_{S_n}) \ge \log h_n\]
%     for all $n$. Of course, this would imply that
%     \[h^A(T) = \sup_\xi h^A(T,\xi) = +\infty.\]
% \end{rem}

Theorem~\ref{seq_ent_blowup} may be generalised slightly to accommodate sequences which contain a large subsequence that dilates to infinity. More specifically, let $J = \{n_k\} \subseteq \NN$ be a strictly increasing sequence. We say that the sequence $A=\{t_n\}$ \emph{dilates to infinity on $J$} if the subsequence $A_J=\{t_{n_k}\}$ dilates to infinity. By a large subsequence $J \subseteq \NN$ we mean a subsequence of positive lower density $\underline{d}(J)$.

% \begin{cor} \label{seq_ent_flex_cor}
%     A rank one system $(X,\cB,\mu,T)$ with arbitrarily large sequence entropy $h^A(T)$ exists if $A=\{t_n\}$ is instead a sequence that dilates to infinity on $J \subseteq \NN$ with positive lower density.
% \end{cor}

\begin{cor} \label{seq_ent_blowup.2}
Let $A$ be a sequence of positive integers that dilates to infinity on a subsequence of positive lower density, then there exists a rank one system $(X,\cB,\mu,T)$ such that $h^A(T)=+\infty$.
\end{cor}

In \cite{Krug72}, a formula for sequence entropy is given by $h^A(T) = K(A) \cdot h(T)$ where
\[K(A) = \lim_{r \to \infty} \varlimsup_{n \to \infty} \frac{1}{n}\#\set{t_i+j}{1 \le i \le n, -r \le j \le r}\]
detects the size of gaps between successive terms in $A = \{t_n\}$ as $n$ tends to infinity. In cases where $K(A)<+\infty$ such as in linear sequences $A=\{an+b\}$, this forces $h^A(T)=0$ for all zero entropy systems. In light of Corollary \ref{seq_ent_blowup.2} one may ask the following:

\begin{qn} \label{qn:Krug-Newton}
    Is the condition $K(A) = +\infty$ equivalent to $A$ dilating to infinity on a subsequence of positive lower density? If not, for a sequence $A$ that does not satisfy the latter, does there exist a measure-preserving system $(X,\cB,\mu,T)$ such that $h^A(T)>0$?
\end{qn}

\medskip

The more intricate part of Question~\ref{qn:seq_ent_flex} lies in the positive but finite case. We will first show that for subexponential but superlinear sequences, there is a lowest threshold for the rate of growth of heights between each subsequent tower in a rank one system, below which necessarily gives rise to zero sequence entropy.

\begin{thm} \label{zero_seq_ent}
Let $(X,\mathcal{B},\mu,T)$ be a rank one system of finite measure approximated by a sequence of towers $S_n$ with base $F_n$ and height $h_n$.
\begin{enumerate}[label=(\roman*)]
    \item \label{zero_seq_ent.1} Let $A=\{t_n\}$ be a subexponential sequence. If \[\varlimsup_{n \to \infty}\frac{\log h_{n+1}}{\log h_{n}} < \infty,\] then $h^A(T)=0$.
    \item \label{zero_seq_ent.2} Let $A=\{t_n\}$ where $t_n = O(n^\alpha)$ for some $\alpha>1$. If there exists $\beta \in (0,\frac{1}{\alpha-1})$ such that \[\varlimsup_{n \to \infty}\frac{\log h_{n+1}}{h_{n}^{\beta}}<\infty,\] then $h^A(T) = 0$.
    \item \label{zero_seq_ent.3} Let $A=\{t_n\}$ where $t_n = \lfloor Cn(\log n)^\alpha\rfloor$ for some $C,\alpha>0$. If there exists $\beta \in (0,\frac{1}{\alpha})$ such that \[\varlimsup_{n \to \infty}\frac{\log h_{n+1}}{e^{h_{n}^{\beta}}}<\infty,\] then $h^A(T) = 0$.
\end{enumerate}
\end{thm}

Theorem~\ref{zero_seq_ent} relies on upper bounds from a combinatorial counting of all possible orbits with respect to a carefully chosen sequence of optimal towers based on orbit length. The bounds for $\alpha$ in \ref{zero_seq_ent.2} and \ref{zero_seq_ent.3} are sharp, as formulated in Theorem~\ref{pos_seq_ent} below.

\begin{thm} \label{pos_seq_ent}
    There is a rank one system $(X,\cB,\mu,T)$ with sequence of tower heights $h_n$ such that $h^A(T)>0$, in addition to either one of the following constraints:
    \begin{enumerate}[label=(\roman*)]
        \item \label{pos_seq_ent.1}  For some $\alpha, \beta > 0$ where $\alpha \ge 1+1/\beta$, $A=\lfloor n^\alpha \rfloor$ and
        \[\varlimsup_{n \to \infty}\frac{\log h_{n+1}}{h_{n}^{\beta}}<\infty;\]
        \item \label{pos_seq_ent.2} For some $\alpha,\beta > 0$ where $\alpha \ge 1/\beta$, $A=\lfloor n(\log n)^\alpha \rfloor$ and
        \[\varlimsup_{n \to \infty}\frac{\log h_{n+1}}{e^{h_{n}^{\beta}}}<\infty.\]
    \end{enumerate}
\end{thm}

Finally, we provide a partial resolution to Question~\ref{qn:seq_ent_flex} for polynomial sequences of the form $t_n=\lfloor n^\alpha \rfloor$. These systems have positive but finite sequence entropy and are therefore atypical in the sense of \cite{Ryzhikov21}, yet are difficult to construct explicitly due to their harsh demand on both the growth rate of tower heights and the complexity of spacer levels. The critical thresholds on tower heights mark a ``Goldilocks zone'' for such systems to appear in relative abundance.

\begin{thm} \label{weak_flex_seq_ent}
    Let $\beta > 0$ and $t_n=\lfloor n^\alpha \rfloor$ where $\alpha = 1+1/\beta$.
    \begin{enumerate}[label=(\roman*)]
        \item \label{small_flex} For any $m > 0$, there is a rank one system $(X,\cB,\mu,T)$ such that $0<h^A(T)<m$;
        \item \label{large_flex} For any integer $L \ge 2$, there is a rank one system $(X,\cB,\mu,T)$ such that \[C_1 \log L \le h^A(T) \le C_2 \log L,\]
        where $C_1=\frac{\beta}{1+\beta}, C_2=(2\beta)^{\frac{1}{1+\beta}}$ depends only on $\beta$.
    \end{enumerate}
\end{thm}

Note that we have $C_1 < C_2$ for any given $\beta \in (0,+\infty)$. Part~\ref{large_flex} implies that for large values of $M>0$, there exists zero entropy systems that achieve a sequence entropy of $M$ up to a multiplicative factor of around $\sqrt{C_2/C_1}$, which has an upper bound of $\sqrt{3}$ when $\beta=0.5$ and converges to 1 as $\beta \to \infty$. This may serve as heuristic evidence for the full flexibility result of sequence entropy along superlinear sequences, at least for sequences with some regularity.

\subsection*{Structure of the paper:} In Section \ref{sec:definitions} we provide a few definitions and fix some notations. In Section \ref{sec:lower_bounds_prob} and \ref{sec:generic_stacks} we lay the groundwork for Section \ref{sec:seq_ent_blowup} where we prove the blow-up result of sequence entropy by randomizing the cut-and-stack procedure of a rank one system. Section \ref{sec:zero_seq_ent} investigates when sequence entropy necessarily vanishes, and Section~\ref{sec:weak_seq_ent_flex} establishes the partial result on the flexibility of sequence entropy along polynomial sequences.

\subsection*{Acknowledgement} This work is supported by the Singapore International Graduate Award (SINGA) Scholarship from the National University of Singapore funded by the Ministry of Education. The author would like to thank Daren Wei for fruitful discussions and his warm encouragement throughout the writing of this paper. The author is grateful to Valerii V. Ryzhikov for his helpful comments on recent developments in this field.

\section{Preliminaries} \label{sec:definitions}

\subsection{Sequence entropy and symbolic dynamics}

Let $(X,\cB,\mu,T)$ be a measure-preserving system and let $\xi=\{X_1,\cdots,X_m\}$ be a measurable partition. The Shannon entropy of $\xi$ is given by
\[H(\xi)=\sum_{i=1}^mf(\mu(X_i))\]
where $f(0)=0$ and $f(x)=-x \log x$ for $x \in (0,+\infty)$. Recall that for any strictly increasing sequence of non-negative integers $A=\{t_n\}$, the sequence entropy of $T$ with respect to $\xi$ along $A$ is defined by
\[h^A(T,\xi) = \varlimsup_{n \to \infty} \frac{1}{n}H\left(\bigvee_{k=1}^n T^{-t_k}\xi\right).\]

Any finite partition $\xi=\{X_1,\cdots,X_m\}$ of a measure-preserving system $(X,\cB,\mu,T)$ can be equipped with a \textit{coding function} $\phi_\xi : X \to \Sigma = \{1,\cdots,m\}$ so that $\phi_\xi(x)=i$ if and only if $x \in X_i$. Given any increasing sequence of non-negative integers $A=(t_n)_{n=1}^\infty$, we define the \textit{sequence coding map} $\phi_{\xi,A,n}:X \to \Sigma^n$ where \[\phi_{\xi,A,n}(x)=\phi_\xi(T^{t_1}x) \cdots \phi_\xi(T^{t_n}x).\] In particular, $\phi_{\xi,n}$ denotes the default coding map where we set $t_n=n-1$. The string $\phi_\xi(T^{t_1}x) \cdots \phi_\xi(T^{t_n}x)$ is called a $\phi_{\xi,A,n}$\textit{-word}, or a $\phi_{\xi,A}$\textit{-word of length} $n$ in $X$.

A partition $\xi'=\{X_1',\cdots,X_l'\}$ is a refinement of $\xi$ if for each $X_i' \in \xi'$ there exists $X_j \in \xi$ such that $X_i' \subseteq X_j$. Given any refinement $\xi'$ of $\xi$, $\phi_{\xi}$ is well-defined as a function on $\xi'$.

\subsection{Constructive geometric definition of rank one systems:} \label{ssec:cg}

This paper adopts the constructive geometric definition of rank one systems (with spacers) given in \cite{Ferenczi97}. Recall that a (\emph{Rokhlin}) \emph{tower} in $X$ is a disjoint union of \textit{level sets} $S = \bigsqcup_{k=0}^{h-1}T^kF$ with \textit{base} $F \in \mathcal{B}$ and \textit{height} $h \in \ZZ_{\ge 0}$, which naturally induces a partition $\xi_S = \{F,TF,\cdots,T^{h-1}F,S^c\}$ of $X$.

\begin{defn}\label{rank_one_system}
    A measure-preserving system $(X,\cB,\mu,T)$ is a \emph{rank one system} if it satisfies the following properties:
    \begin{enumerate}
        \item There exists $q_n \in \NN$ for all $n \in \NN$, $a_{n,i} \in \ZZ_{\ge 0}$ for all $n,i \in \NN$ where $1 \le i \le q_n-1$, such that the sum $\sum_{n=1}^\infty a_n/h_{n+1}$ converges,
        where $a_n=\sum_{i=1}^{q_n-1}a_{n,i}$ and $h_n$ is the sequence defined by $h_0=1$ and $h_{n+1}=q_nh_n+a_n$;
        \item There exists measurable subsets $F_n$ for all $n \in \NN$, $F_{n,i}$ for all $n \in \NN, 1 \le i \le q_n$, and $C_{n,i,j}$ for all $n \in \NN, 1 \le i \le q_n-1, 1 \le j \le a_i$, such that
        \begin{enumerate}
            \item For any given $n \in \NN$, $\set{T^kF_n}{0 \le k \le h_n-1}$ is pairwise disjoint;
            \item $\set{F_{n,i}}{1 \le i \le q_n-1}$ forms a partition of $F_n$;
            \item For all $1 \le i \le q_n-1$, $T^{h_n}F_{n,i}=\begin{cases}
                C_{n,i,0} & (a_i>0); \\
                F_{n,i+1} & (a_i=0);
            \end{cases}$
            \item For all $1 \le i \le q_n-1$, $TC_{n,i,j}=\begin{cases}
                C_{n,i,j+1} & (j<a_i); \\
                F_{n,i+1} & (j=a_i);
            \end{cases}$
            \item $F_{n+1}=F_{n,1}$.
        \end{enumerate}
    \end{enumerate}
\end{defn}

\begin{rem}
The above definition of a rank one system is often referred to as the \emph{cut-and-stack procedure}. Informally, the procedure begins by fixing an initial tower $S_1=\bigsqcup_{k=0}^{h_1-1}T^kF_1$. This is followed by an indefinite iteration of a series of steps described below:
\begin{enumerate}
    \item Let $S_n=\bigsqcup_{k=0}^{h_n-1}T^kF_n$ be the $n$-th tower constructed in the cut-and stack procedure. For some $q_n \in \NN$, divide the base $F_n$ into $q_n$ disjoint subsets of equal measure, denoted by $F_{n,i}$ where $i=1,\cdots,q_n$.
    \item Fix a sequence of non-negative integers $a_{n,1},\cdots,a_{n,q_n} \in \ZZ_{\ge 0}$. Define a collection of sets $C_{n,i,j}$ for all integers $i,j$ that satisfies $1 \le i \le q_n-1$ and $0 \le j \le a_{n,i}-1$, so that $\mu(C_{n,i,j})=\mu(F_{n,i})$ for every such set.
    \item Define measure-preserving bijections between $F_{n,i}$ and $C_{n,i,j}$ such that they satisfy the required gluing properties in Definition \ref{rank_one_system}. The resulting construction is a tower $S_{n+1}$ with base $F_{n+1}=F_{n,1}$ and height $h_{n+1}=q_nh_n+a_n$.
\end{enumerate}
\end{rem}

The collection of integers $(q_n;a_{n,i})$ will sometimes be referred to as the \emph{stacking data} of the rank one system as these parameters determine the system up to measure-theoretic isomorphism. A different set of stacking data may give rise to the same system, for instance, by taking a subsequence $\{S_{n_k}\}$ of approximating towers.

\subsection{Refinement of induced partitions} \label{ssec:refinement_induced}

Let $S=\bigsqcup_{k=1}^{h-1}T^kF$ be a tower in a measure-preserving system $X$ with induced partition $\xi$ and coding function $\phi:X \to \Sigma = \{0,1,\cdots,h\}$ where $\phi(S^c)=0$ and $\phi(T^kF)=k+1$ for $k=0,\cdots,h-1$.

In practice, it will be helpful to have a coding map which distinguishes the spacer levels of a rank one system. Naturally, the Kakutani skyscraper with base $F$ induces a countable partition $\xi'$ that refines $\xi$:
\begin{multline*}
    \xi'=\left\{\left(\bigcup_{k=0}^\infty T^kF\right)^c,F,TF,\cdots,T^{h-1}F, \right. \\
    \left. T^hF \setminus \left(\bigcup_{k=0}^{h-1} T^kF\right), T^{h+1}F \setminus \left(\bigcup_{k=0}^{h} T^kF\right), \cdots\right\}.
\end{multline*}

One may then define the coding function $\phi':X \to \Sigma'=\NN_0$ which returns 0 outside the skyscraper, and $k+1$ on the $k$-th level of the skyscraper for $k=0,1,\cdots$. If $(X,\cB,\mu,T)$ is ergodic, the Kakutani skyscraper covers almost all points in $X$ and the skyscraper complement does not constitute an atom of $\xi'$. Furthermore, if $X$ is a rank one system with spacer levels $a_{n,i}$ uniformly bounded by $L$, then $\xi'$ is a finite partition with $h+L$ atoms.

In the constructive geometric definition of rank one systems as outlined in \ref{ssec:cg}, the coding map $\phi_{\xi'}$ of the natural refinement of the induced partition $\xi=\xi_{S_1}$ makes sense on each $S_n$ even when $(X,\cB,\mu,T)$ has not been fully defined. If $T^kF_n$ is a level set of $S_n$ where $0 \le k \le h_n-1$, we have $\phi_{\xi'}(T^kF_n)=k-k_0+1$ where $k_0 \le k$ is the largest integer that satisfies $T^{k_0}F_n \subseteq F_1$.

\subsection{Paper-specific notations}\label{ssec:notations}

We now fix the precise notations used throughout the rest of the paper. Let $S_n,h_n,F_n,q_n,a_{n,i}$ be as in Definition \ref{rank_one_system}, and let $\xi_n=\xi_{S_n}$ be the induced partitions of the towers. The coding function of $\xi_n$ is defined as the function $\phi_{\xi_n}:X \to \Sigma_n$, where $\Sigma_n:=\{0,1,\cdots,h_n\}$, $\phi_{\xi_n}(S_n^c)=0$ and $\phi_{\xi_n}(T^kF_n) = k+1$ for $k=0,\cdots,h_n-1$.

The induced partitions $\xi_n$ form a sequence of refinements that approach the point partition $\veps$. Given any $1 \le r \le n$, $\phi_{\xi_r}$ is well-defined as a function on the level sets $\{T^kF_n\}$ in $S_n$ and takes the value $0$ on $S_n^c$. Denote by $\Phi_r(S_n) := \phi_{\xi_r,h_n}(F_n)$ the $\phi_{\xi_r}$-word that encodes $S_n$. The cut-and-stack procedure that defines the rank one system can be described by the recursive relation
\begin{equation} \label{cut_and_stack_relation}
    \Phi_r(S_{n+1})=\Phi_r(S_n)0^{a_{n,1}}\Phi_r(S_n)0^{a_{n,2}} \cdots \Phi_r(S_n)0^{a_{n,q_n-1}}\Phi_r(S_n)
\end{equation}
where $1 \le r \le n$ and $0$ corresponds to the spacer levels in $S_{n+1} \setminus S_n$. Note that the number of spacer levels $a_{n,i}$ is independent of the choice of the coding function $\phi_{\xi_r}$.

% \medskip

% The ensuing calculations on sequence entropy rely on restrictions into the towers that gradually exhausts a rank one system. We will give a few additional definitions that cater to our needs.

% \begin{defn} \label{vocab_approx}
%     Let $S = \bigsqcup_{k=0}^{h-1} T^kF$, $A = \{t_n\}$, $n \in \NN$, and let $\xi$ be a partition refined by $\xi_S$. 
%     \begin{itemize}
%         \item The \emph{$\phi_{\xi,A,n}$-vocabulary of $S$} is the set $\cV$ of all $\phi_{\xi,A,n}$-words fully contained in $S$, i.e. \[\cV = \set{\phi_{\xi,A,n}(T^kF)}{0 \le k \le h-t_n-1}.\]
%         \item The \emph{approximate sequence entropy restricted on $S$}, denoted by $H_S$, is defined by \[H_S\left(\bigvee_{k=1}^n T^{-t_k}\xi\right) = \sum_{w \in \cV}f\left(\mu\phi_{\xi,A,n}^{-1}(w)\right)\] where $f(0)=0$ and $f(x)=-x\log x$ when $x > 0$.
%     \end{itemize}
% \end{defn}

\section{Lower bounds and the probabilistic method} \label{sec:lower_bounds_prob}

% In a rank one system with positive sequence entropy, the sequence of towers $S_n$ should allow for gaps between subsequent terms in $A$ to surpass the scale of $h_n$, and furthermore, allows sufficient time still for a rich variety of words of that scale to appear.

Let $A=\{t_n\}$ be a strictly increasing but subexponential sequence of non-negative integers. Let $S_n = \bigsqcup_{k=0}^{h_n-1} T^kF_n$ be a sequence of towers approximating the rank one system $(X,\cB,\mu,T)$ defined recursively via a cut-and-stack procedure described by Equation (\ref{cut_and_stack_relation}).

\medskip

Throughout Sections \ref{sec:lower_bounds_prob} to \ref{sec:seq_ent_blowup}, for simplicity we restrict the number of spacer levels $a_{n,i}$ to be either 0 or 1. Note that this guarantees \[\sum_{n=1}^\infty \frac{a_n}{h_{n+1}} \le \sum_{n=1}^\infty \frac{q_n}{q_nh_n} < \infty\] as required in Definition \ref{rank_one_system}.

\medskip

We first state a preliminary lower bound for the entropy calculated on a large disjoint union of small subsets:

\begin{prop} \label{disjoint_lower_bound}
    Let $X$ be a finite measure space and let $E=\bigsqcup_{i}E_i$ be a finite or countable disjoint union of measurable subsets such that $\mu(E_i) \le \veps\mu(X)$ for some $\veps>0$, then
    \[\sum_{i}f(\mu(E_i)) \ge -\mu(E) \log (\veps\mu(X)).\]
\end{prop}

\begin{proof}
    This follows from the concavity bound $f(x) \ge \frac{f(a)}{a}x$ for $x \in [0,a]$ where we take $a=\veps \mu(X)$.
\end{proof}

To obtain a non-vanishing sequence entropy, Proposition \ref{disjoint_lower_bound} hints that, at each step from $S_n$ to $S_{n+1}$, a natural number $N_n$ and a sequence $a_{n,i}$ should be chosen in a way so that each $\phi_{\xi_rA,N_n}$-word $w$ does not constitute a large measure in $X$. The existence of such a sequence will be demonstrated by interpreting $a_{n,i}$ as a sequence of independent random variables with values in $\{0,1\}$ and showing that each word do not appear too frequently in generic cases when the number of slices $q_n$ is sufficiently large.

\medskip

We now prove the key probabilistic lemma used to control the measure of each distinguishable $A$-orbit. In essence, for $1 \le r \le n$, the $\phi_{\xi_r}$-word $\Phi_r(S_n)$ is composed of numerous repetitions of $\Phi_r(S_r)=12 \cdots h_r$ intermittently separated by spacer symbols $0$; all such tower-encoding words will be encompassed by the word $w$ in the statement and proof of this lemma.

\medskip

If we interpret $a_{n,i}$ in Equation (\ref{cut_and_stack_relation}) as independent Bernoulli trials, symbols that are far apart in $\Phi_r(S_{n+1})$ will be nearly independent to one another. The near independence will be proven by establishing a recursive relation that resembles a mixing Markov chain. This result will then be applicable to sufficiently long orbits as $A$ dilates to infinity.

\begin{prop}\label{Bernoulli_pushforward}
    Let $h \ge 2$, $\Sigma=\{1,\cdots,h\}$, $v=12 \cdots h \in \Sigma^h$ and \[w=v0^{b_1}v0^{b_2}\cdots v0^{b_{g-1}}v \in \Sigma^{H}\] where $H=gh+b$, $g \ge 1$, $b=\sum_{i=1}^{g-1}b_i$ and $b_i \in \{0,1\}$ for $i=1,\cdots,g-1$. Let $(\Omega,\cF,\PP)$ be the Bernoulli process equipped with the Bernoulli shift $\sigma$, and define the map $\iota:\Omega \to \Sigma^\NN$ where each $a=(a_i)_{i \in \NN} \in \Omega$ corresponds to an infinite string \[\iota(a)=w0^{a_1}w0^{a_2}\cdots \in \Sigma^\NN.\]
    \begin{enumerate}[label=(\roman*)]
        \item \label{Bernoulli_pushforward.1} For any $l_0,l_1 \in \Sigma$ and $n \in \NN$ where $\PP(\iota(a)_n=l_0)>0$, we have \[\lim_{s \to \infty}\PP(\iota(a)_{n+s}=l_1 \mid \iota(a)_n=l_0) < \frac{1}{h}.\]
        \item \label{Bernoulli_pushforward.2} There exists $s \in \NN$ such that for all $n_0,n_1,\cdots,n_N \in \NN$ such that $n_{i+1}-n_{i} \ge s$ for $i=0,1,\cdots,N-1$, we have \[\PP(\iota(a)_{n_0} \cdots \iota(a)_{n_N}=l_{0} \cdots l_{N}) \le \frac{1}{h^N}\] for any $l_0,l_1,\cdots,l_N \in \Sigma$.
    \end{enumerate}
\end{prop}

\begin{proof}
    It is easy to see that \ref{Bernoulli_pushforward.1} implies \ref{Bernoulli_pushforward.2}, thus we prove \ref{Bernoulli_pushforward.1}. Denote $\tilde{w}=\iota(a)$, then since $\tilde{w}_n=l_1$ if and only if $\tilde{w}_{n-l_1+1}=1$ for $l_1=2,\cdots,h$, it suffices to show that the limits exist and \[0<\lim_{s \to \infty}\PP(\tilde{w}_{n+s}=0 \mid \tilde{w}_n=l_0)<\frac{1}{h}\] for any $l_0 \in \Sigma$, as the limit for $l_1=1,\cdots,h$ would then be equal to some shared value strictly less than $\frac{1}{h}$. For any $s \ge H+1$, let $q : \Omega \to \NN$ be the random variable which corresponds to the last spacer level that determines $\tilde{w}_{n+s}$, i.e. $q=q(a)$ is the largest natural number that satisfies $qH+\sum_{i=1}^{q-1}a_i<n+s$. We have
    \begin{align*}
        & \PP(\tilde{w}_n=l_0,\tilde{w}_{n+s}=l_1) \\
        = & \frac{1}{2}\left(\PP(\tilde{w}_n=l_0,\tilde{w}_{n+s}=l_1 \mid a_q=0)+\PP(\tilde{w}_n=l_0,\tilde{w}_{n+s}=l_1 \mid a_q=1)\right) \\
        = & \frac{1}{2}\left(\PP(\tilde{w}_n=l_0,\tilde{w}_{n+s-H}=l_1)+\PP(\tilde{w}_n=l_0,\tilde{w}_{n+s-H-1}=l_1)\right)
    \end{align*}
    and therefore $p_s=\frac{1}{2}(p_{s-H}+p_{s-H-1})$ where $p_s:=\PP(\tilde{w}_{n+s}=l_1 \mid \tilde{w}_n=l_0)$. Hence,
    \[\begin{pmatrix}
        p_s \\ p_{s-1} \\ \vdots \\ p_{s-H}
    \end{pmatrix} = 
    A\begin{pmatrix}
        p_{s-1} \\ p_{s-2} \\ \vdots \\ p_{s-H-1}
    \end{pmatrix}\]
    for $s \ge H+1$, where
    \[A=\begin{pmatrix}
        & & \frac{1}{2} & \frac{1}{2} \\
        1 \\
        & \ddots \\
        & & 1
    \end{pmatrix}\]
    is irreducible and aperiodic (see \cite{PetersenErgodicTheory}, Theorem 2.5.6), hence the limit of $A^n$ exists and each row converges to the stationary distribution of $A$ given by \[\pi=\begin{pmatrix}\frac{2}{2H-1} & \cdots & \frac{2}{2H-1} & \frac{1}{2H-1}\end{pmatrix}.\] This shows that for each $l_1 \in \Sigma$, the limit of $p_s$ exists and \[\lim_{s \to \infty}p_s = \frac{2}{2H-1}\left(p_{H}+p_{H-1}+\cdots+p_1\right)+\frac{1}{2H-1}p_0>0.\]
    
    For $l_1=0$, note that as soon as the initial position $\tilde{w}_n$ is fixed, the word $\tilde{w}_n\tilde{w}_{n+1} \cdots \tilde{w}_{n+H}$ is determined by a single spacer level $a_{i}$, except for the case where $\tilde{w}_n$ itself coincides with a spacer level $0^{a_i}$ with $a_i=1$. In any case, regardless of the value of $a_i$, any two `0's can only appear in $\tilde{w}$ with at least $h$ symbols in between. Therefore, we have \[\begin{pmatrix}
        p_{H} \\ p_{H-1} \\ \cdots \\ p_0
    \end{pmatrix} = \frac{1}{2}\begin{pmatrix}
        p_{H}^{(0)} \\ p_{H-1}^{(0)} \\ \cdots \\ p_0^{(0)}
    \end{pmatrix}+\frac{1}{2}\begin{pmatrix}
        p_{H}^{(1)} \\ p_{H-1}^{(1)} \\ \cdots \\ p_0^{(1)}
    \end{pmatrix}\] where both vectors $(p_i^{(0)}),(p_i^{(1)})$ have entries in $\{0,1\}$ and each pair of `1's is separated by at least $h$ `0's. Multiplying by $\pi$ gives the limit \[\lim_{s \to \infty} p_s \le \frac{2}{2H-1}\left\lceil\frac{H+1}{h+1}\right\rceil \le \frac{2g}{2H-1}<\frac{1}{h},\] at least for $H=gh+b$ with $g \ge 2$ and $1 \le b \le g-1$. For the particular case where $b=0$ we instead have $\lim_{s \to \infty}p_s=1/(2H-1)<1/h$.
\end{proof}

\section{Genericness of high-complexity stacks} \label{sec:generic_stacks}

We now establish the main lemma, Proposition~\ref{generic_stacks}, which describes the genericness of high-complexity stacks with respect to a sequence $A=\{t_n\}$ that dilates to infinity. This will serve as the inductive step of the desired construction in the proof of Theorem~\ref{seq_ent_blowup}. The following version of Hoeffding's inequality for $m$-dependent variables will be used and can be found in his original paper \cite{Hoeffding63}.

\begin{prop}\label{m-Hoeffding}
    Let $\{X_i\}_{i \in \NN}$ be a sequence of $m$-dependent variables, i.e. the joint distributions of $\{X_i\}_{i \le k}$ and $\{X_i\}_{i > k+m}$ are independent for all $k\in\NN$. Suppose that $0 \le X_n \le 1$ almost surely for all $n$, and define $\bar{X}_n=\frac{1}{n}\sum_{k=1}^nX_k, \mu_n=\EE \bar{X}_n$. Then, for any $t > 0$, we have
    \[\PP(\bar{X}_n-\mu_n \ge t) \le \exp(-2\lfloor n/m \rfloor t^2).\]
\end{prop}

\subsection{Context and outline of proof}\label{ssec:generic_stacks_prep}

Recall the descriptions and notations related to the approximating towers described in Subsection~\ref{ssec:notations}. Suppose that $S_1,\cdots,S_n$ have already been defined via a cut-and-stack procedure described by Equation (\ref{cut_and_stack_relation}), with $a_{r,i} \in \{0,1\}$ for $1 \le r \le n-1$ and $1 \le i \le q_r-1$.

\medskip

As we hope to achieve a high sequence entropy for partitions induced by all of its towers, the $n$-th instance of the stacking data $(q_n;a_{n,i})$ will be obtained via a diagonal argument. Repeated iterations show that for $r=1,\cdots,n$, there exists $g(n,r) \in \NN$ and $b_{r,1},\cdots,b_{r,g(n,r)-1} \in \{0,1\}$ such that the $\phi_r$-word which encodes $S_n$ is given by
\[\Phi_r(S_n)=v_r0^{b_{r,1}}v_r0^{b_{r,2}} \cdots v_r0^{b_{r,g(n,r)-1}}v_r \in \Sigma_r^{h_n}\]
where $v_r=\Phi_r(S_r)=12 \cdots h_r$. A random binary sequence $a=(a_i)_{i \in \NN}$ can be used to define an infinite string
\[\iota_{n,r}(a)=\Phi_r(S_n)0^{a_1}\Phi_r(S_n)0^{a_2} \cdots \in \Sigma_r^\NN,\]
see Figure~\ref{fig:generic_stacks}. The subsequent tower, $S_{n+1}$, can then be constructed once the stacking data $(q_n;a_{n,i})$ are determined. The goal is to then show, under the sampling of all $A$-orbits of some fixed length $N_n$, that $S_{n+1}$ is sufficiently complex for at least one choice of the stacking data.

% Let $S_1,\cdots,S_n$ be a sequence of approximating towers in $(X,\cB,\mu,T)$ where {\color{red}$S_i=\bigsqcup_{k=0}^{h_i-1}T^kF_i$.} Denote $\xi_i=\xi_{S_i},\Sigma_i=\{0,1,\cdots,h_i\}, \phi_i=\phi_{S_i}:X \to \Sigma_i$ for $1 \le i \le n$ and {\color{blue}$\Phi_r(S_i)=\phi_{\xi_r,h_i}(F_i) \in \Sigma_r^{h_i}$} for $1 \le r \le i \le n$. Suppose that the sequence of towers {\color{red}$S_1,\cdots,S_n$ is defined via a cut-and-stack procedure described by \[\Phi_r(S_{i+1})=\Phi_r(S_i) 0^{a_{i,1}}\Phi_r(S_i) 0^{a_{i,2}} \cdots \Phi_r(S_i) 0^{a_{i,q_i-1}}\Phi_r(S_i)\] for $1 \le r \le i \le n-1$ where $a_{i,j} \in \{0,1\}$.}

\begin{figure}[htp!]
\centering
\begin{tikzpicture}
\foreach \x in {0,2,6}
    {
    \foreach \y in {0,0.5,1.5}
        \draw (0,\x+\y) rectangle (1,\x+\y+0.3);
    \foreach \y in {0.4,0.9,1.4}
        \draw[red] (0,\x+\y) -- (1,\x+\y);
    \draw (0.5,\x+1.25) node {$\vdots$};
    }
\foreach \x in {1.9,3.9,5.9}
    \draw[red] (0,\x) -- (1,\x);

\foreach \x in {0,2,6}
    \draw (1.5,\x+0) rectangle (2.5,\x+1.8);
\foreach \x in {1.9,3.9,5.9}
    \draw[red] (1.5,\x) -- (2.5,\x);

\draw (0.5,5) node {$\vdots$};
\draw (2,5) node {$\vdots$};
\draw (3,3.9) node {$\cdots$};
\draw (3.5,0) rectangle (4.5,7.8);

\foreach \x in {0,3,6}
    {
    \foreach \y in {7,8.5,10.5}
        \draw (\y,\x) rectangle (\y+1,\x+2.6);
    \draw (10,\x+1.3) node {$\cdots$};
    \draw[red,dashed] (6.5,\x+2.8) -- (12,\x+2.8);
    }
\draw (9.25,9.25) node {$\vdots$};

\draw[dotted] (-0.3,-0.8) rectangle (4.8,8.1);
\draw[dotted] (6.7,-0.1) rectangle (11.8,2.7);
\draw[dotted] (4.8,1.3) -- (6.7,1.3);

\draw (0.5,0) node[anchor=north] {$\Phi_1(S_n)$};
\draw (2,0) node[anchor=north] {$\Phi_2(S_n)$};
\draw (4,0) node[anchor=north] {$\Phi_n(S_n)$};
\draw (7.5,0) node[anchor=north] {$\iota_{n,1}(a)$};
\draw (9,0) node[anchor=north] {$\iota_{n,2}(a)$};
\draw (11,0) node[anchor=north] {$\iota_{n,n}(a)$};
\end{tikzpicture}
\caption{Construction of random sequences, to be read bottom-up. Each box in $\Phi_r(S_n)$ refers to the word $\Phi_r(S_r)=12 \cdots h_r$. Red lines correspond to spacer symbols $0$. Solid red lines represent the already determined $0^{b_{r,i}}$ levels and may or may not correspond to an existent spacer symbol. Dashed red lines represent the randomised levels $0^{a_i}$.}
\label{fig:generic_stacks}
\end{figure}
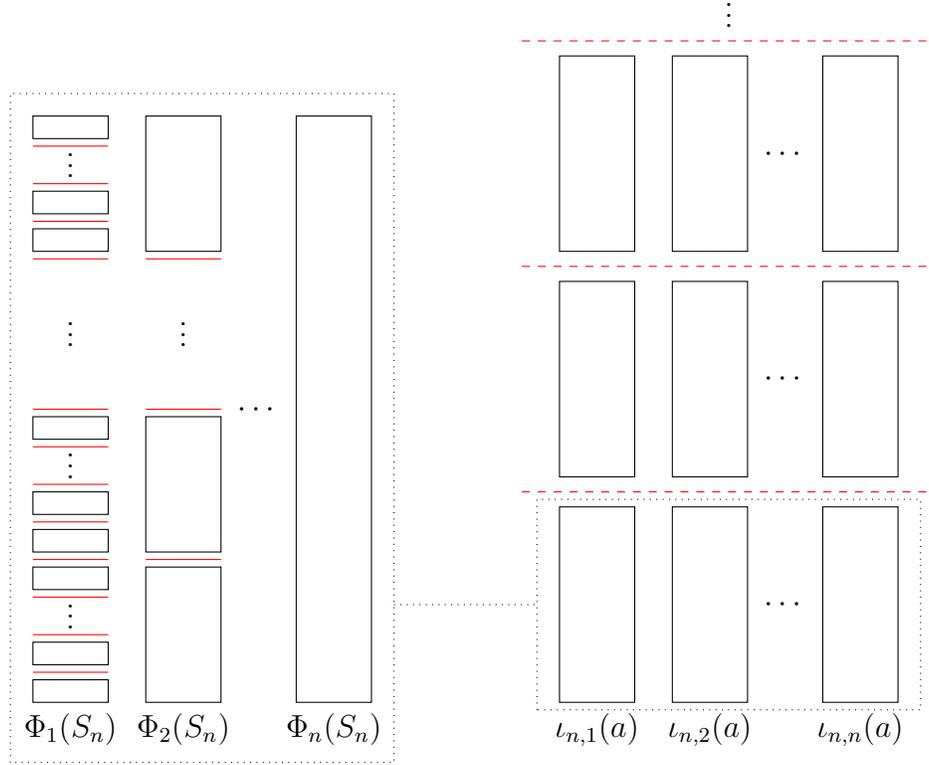

\medskip

Figure \ref{fig:hierarchy} illustrates the key parameters used in the ensuing argument. We first choose the sampling length $N_n$ large enough so that two conditions are met simultaneously:
\begin{enumerate}
    \item the gaps between each term in $A=\{t_n\}$ grow large enough for consecutive symbols to be nearly independent; and
    \item $N_n$ satisfies an inequality of the form $P(N_n)<1$, where $P(N)$ is independent of $q_n$ and $P(N) \to 0$ as $N \to \infty$.
\end{enumerate}

The number of slices $q_n$ is then chosen so that the probability of obtaining an exceedingly ordered sequence of spacings $a_1,\cdots,a_{q_n-1}$ is below $P(N_n)$ by Hoeffding's inequality. The scale of $q_n$ will turn out to be at least exponential in $N_n$, so that the topmost $t_{N_n}$ level sets with $A$-orbits that escape $S_{n+1}$ are negligible.

\medskip

We now shift our focus to the precise measurement of sequence entropy. In principle, one would need to consider the frequency of each word $v \in \Sigma_r^{N_n}$ appearing as a $\phi_{\xi_r,A,N_n}$-word for level sets in $S_{n+1}$, which is well-defined as long as the orbit stays within the tower. Denote by $K_n$ the set of indices corresponding to the \emph{non-spacer} levels $T^kF_{n+1}$ such that the $A$-orbit $\{T^{k+t_i}F_{n+1}\}_{i=1}^{N_n}$ stays in $S_{n+1}$; formally, we have
\[K_n = \set{k \in \ZZ}{0 \le k \le h_{n+1}-t_{N_n}-1,T^kF_{n+1} \subseteq S_{n}}.\]
The avoidance of spacer levels will simplify the proof of Proposition~\ref{generic_stacks} and is insignificant even without the assumption $a_{n,i} \in \{0,1\}$: the number of spacer levels $a_n = \sum_{i=1}^{q_n-1} a_{n,i}$ satisfies that $a_n/h_{n+1} \to 0$ since $\sum_{n=1}^\infty a_n/h_{n+1}$ converges as required in Definition~\ref{rank_one_system}.

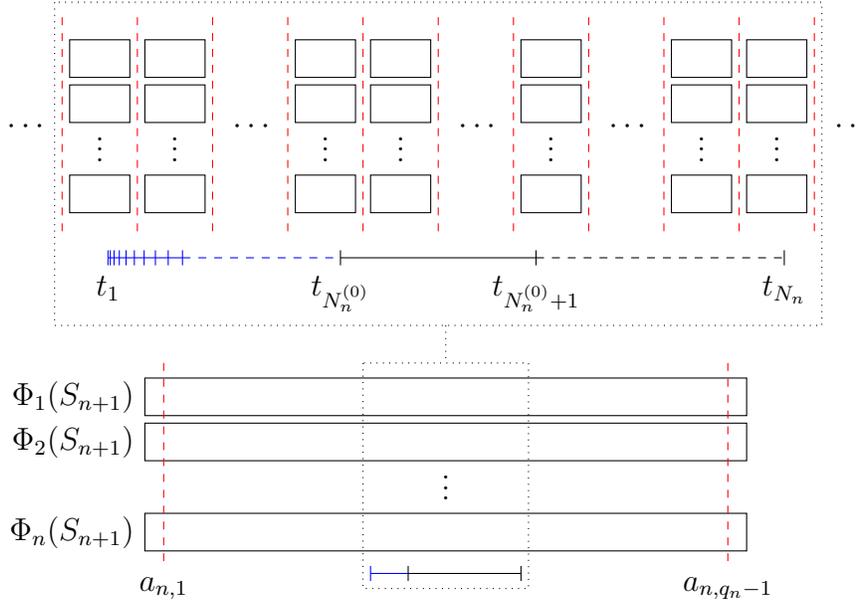
\begin{figure}
\centering
\begin{tikzpicture}
\foreach \x in {1,2,4,5,7,9,10}
    {
    \foreach \y in {0,-0.6,-1.8}
        \draw (\x,\y) rectangle (\x+0.8,\y-0.5);
    \draw (\x+0.4,-1.35) node[anchor=center] {$\vdots$};
    }
\foreach \x in {0,3,6,8,11}
    \draw (\x+0.45,-1.15) node {$\cdots$};
\foreach \x in {1,...,11}
    \draw[red,dashed] (\x-0.1,0.3) -- (\x-0.1,-2.6);

\foreach \x [evaluate = \x as \y using \x^2] in {0.1,0.2,...,1}
    \draw[blue] (1.5+\y,-2.8) -- (1.5+\y,-3);
\foreach \x in {4.6,7.2,10.5}
    \draw (\x,-2.8) -- (\x,-3);
\draw[blue] (1.51,-2.9) -- (2.5,-2.9);
\draw[blue,dashed] (2.5,-2.9) -- (4.6,-2.9);
\draw (4.6,-2.9) -- (7.2,-2.9);
\draw[dashed] (7.2,-2.9) -- (10.5,-2.9);
\draw (1.51,-3) node[anchor=north] {$t_1$};
\draw (4.6,-3) node[anchor=north] {$t_{N_n^{(0)}}$};
\draw (7.2,-3) node[anchor=north] {$t_{N_n^{(0)}+1}$};
\draw (10.5,-3) node[anchor=north] {$t_{N_n}$};

\foreach \x in {-4.5,-5.1,-6.3}
    \draw (2,\x) rectangle (10,\x-0.5);
\draw (6,-5.85) node[anchor=center] {$\vdots$};
\foreach \x in {2.25,9.75}
    \draw[red,dashed] (\x,-4.3) -- (\x,-7);

\draw[blue] (5,-7) -- (5,-7.2);
\draw (5.5,-7) -- (5.5,-7.2);
\draw (7,-7) -- (7,-7.2);
\draw[blue] (5,-7.1) -- (5.5,-7.1);
\draw (5.5,-7.1) -- (7,-7.1);

\draw (2.25,-7) node[anchor=north] {$a_{n,1}$};
\draw (9.75,-7) node[anchor=north] {$a_{n,q_n-1}$};
\draw (2,-4.75) node[anchor=east] {$\Phi_1(S_{n+1})$};
\draw (2,-5.35) node[anchor=east] {$\Phi_2(S_{n+1})$};
\draw (2,-6.55) node[anchor=east] {$\Phi_n(S_{n+1})$};

\draw[dotted] (0.8,0.5) rectangle (11.0,-3.8);
\draw[dotted] (4.9,-4.3) rectangle (7.1,-7.3);
\draw[dotted] (6,-3.8) -- (6,-4.3);
\end{tikzpicture}
\caption{The hierarchy of parameter sizes at each step of the cut-and-stack procedure. Blue hatch marks correspond to the initial symbols of an $A$-orbit that are too predictable, whereas black ones have sufficiently wide gaps in between for adequate mixing.}
\label{fig:hierarchy}
\end{figure}

\subsection{Statement and proof}
\begin{prop}\label{generic_stacks}
Under the framework of Subsection~\ref{ssec:generic_stacks_prep}, there exists $N_n^{(0)} \in \NN$ such that for all sufficiently large $N_n > N_n^{(0)}$, there exists $q_n \in \NN$ dependent on $N_n$ such that one can pick $a_{n,1},\cdots,a_{n,q_n-1} \in \{0,1\}$ for $S_{n+1}$ so that, for $1 \le r \le n$, every $v \in \Sigma_r^{N_n}$ is the $\phi_{\xi_r,A,N_n}$-word of no more than $2h_r^{-(N_n-N_n^{(0)})} \cdot |K_n|$ distinct non-spacer levels in $S_n$.
\end{prop}

\begin{proof}
Let $(\Omega,\cF,\cP)$ be the Bernoulli process and for $r=1,\cdots,n$, define the map $\iota_{n,r}:\Omega \to \Sigma_r^\NN$ where each $a=(a_i)_{i \in \NN} \in \Omega$ corresponds to an infinite string \[\iota_{n,r}(a)=\Phi_r(S_n)0^{a_1}\Phi_r(S_n)0^{a_2} \cdots \in \Sigma_r^\NN.\]
By Proposition \ref{Bernoulli_pushforward}, for $r=1,\cdots,n$ there exists $s_{n,r} \in \NN$ such that for all $k_0, \cdots, k_N \in \NN$ with $k_{i+1}-k_i \ge s_{n,r}$ for $i=0,\cdots,N-1$, we have \[\PP(\iota_{n,r}(a)_{k_0} \cdots \iota_{n,r}(a)_{k_N}=l_0 \cdots l_N) \le \frac{1}{h_r^N}\]
for any $l_0,\cdots,l_N \in \Sigma_r$. Take $s_n=\max_{1 \le r \le n}\{s_{n,r}\}$ and pick $N_n^{(0)} \in \NN$ such that $t_{i+1}-t_i \ge s_n$ for all $i \ge N_n^{(0)}$, then for $N \ge N_n^{(0)}$, $1 \le r \le n$, $v \in \Sigma_r^N$, and $k,j \in \NN$ we have \[\PP(\iota_{n,r}(\sigma^ka)_{j+t_1} \cdots \iota_{n,r}(\sigma^ka)_{j+t_N} = l_1 \cdots l_N) \le \frac{1}{h_r^{N-N_n^{(0)}}}\]
for any $l_1,\cdots,l_N \in \Sigma_r$ (here $\sigma$ denotes the Bernoulli shift map).

For $1 \le r \le n$, $1 \le j \le h_n$ and $v \in \Sigma_r^N$, let $(X^{(r,j)}_k(v))_{k=1}^\infty$ be a sequence of random variables defined as indicator functions on $\Omega$ as follows:
\[X^{(r,j)}_k(v) = \begin{cases}
        1 & (\iota_{n,r}(\sigma^ka)_{j+t_1} \cdots \iota_{n,r}(\sigma^ka)_{j+t_N}=v), \\
        0 & (\text{otherwise}).
    \end{cases}\]

Let $m=\lfloor t_N/h_n \rfloor+2$ so that $h_nm > h_n+t_N$, then the sequence of random variables $(X^{(r,j)}_k(v))_{k=1}^\infty$ is $(m-1)$-dependent, thus by Hoeffding's Inequality (Proposition \ref{m-Hoeffding}) for any $\delta>0, q \in \NN$ we have
\[\PP(\{a \in \Omega: \bar{X}^{(r,j)}_q(v)-\EE \bar{X}^{(r,j)}_q(v) \ge \delta\}) \le \exp(-2\lfloor q/m \rfloor \delta^2)\]
where $\EE\bar{X}^{(r,j)}_q(v) \le h_r^{-(N-N_n^{(0)})}$ for all $q$. Take $\delta = h_r^{-N}, q \ge mh_r^{3N}$, then
\begin{equation} \label{generic_stacks_lemma}
    \PP(\{a \in \Omega: \bar{X}^{(r,j)}_q(v) \ge 2h_r^{-(N-N_n^{(0)})}\}) \le \exp(-2h_r^N).
\end{equation}

% Let $\{i_k\} \subseteq \NN$ be the (almost surely infinite) increasing sequence of indices that satisfies $a_{i_k}=1$, then the sequence $(Y^{(r)}_k(v))_{k=1}^\infty$ defined by
% \[Y^{(r)}_k(v)=\begin{cases}
%     1 & (t_1=0,0\iota_{n,r}(\sigma^{i_k}a)_{t_2} \cdots \iota_{n,r}(\sigma^{i_k}a)_{t_N}=v), \\
%     1 & (t_1>0,\iota_{n,r}(\sigma^{i_k}a)_{t_1} \cdots \iota_{n,r}(\sigma^{i_k}a)_{t_N}=v), \\
%     0 & (\text{otherwise})
% \end{cases}\]
% is also $(m-1)$-dependent and $\EE\bar{Y}^{(r)}_q(v) \le h_r^{-(N-N_n^{(0)})}$, thus for $q \ge mh_r^{3N}$,
% \[\PP(\{\bar{Y}^{(r)}_q(v) \ge 2h_r^{-(N-N_n^{(0)})}\}) \le \exp(-2h_r^N).\]

For any $a \in \Omega$, consider the random tower $S_{n+1}=\bigsqcup_{k=0}^{h_{n+1}-1}T^kF_{n+1}$ with height $h_{n+1}=h_nq+\sum_{i=1}^{q-1}a_i$. Then, for any specific choice of $r \in \{1,\cdots,n\}$ and $v \in \Sigma_r^N$, the probability that $\phi_{\xi_r,A,N_n}(T^kF_{n+1})=v$ holds for more than $2h_r^{-(N-N_n^{(0)})} \cdot |K_n|$ distinct values of $k \in K_n$ is below \[\PP\left(\left\{a \in \Omega:\max\{\bar{X}^{(r,1)}_{q^{(1)}}(v), \cdots,\bar{X}^{(r,h_n)}_{q^{(h_n)}}(v)\} \ge 2h_r^{-(N-N_n^{(0)})}\right\}\right)\] where $q^{(j)}$ is the largest integer such that \[\left((q^{(j)}-1)h_n+\sum_{i=1}^{q^{(j)}-1}a_i+j\right)+t_N \le h_{n+1}.\]
Fix $q=m(h_n^{3N}+1)$, then from
\begin{align*}
    & h_{n+1}-\left((mh_n^{3N}-1)h_n+\sum_{i=1}^{mh_n^{3N}-1}a_i+h_n\right)-t_N \\
    = \ & mh_n + \sum_{i=mh_n^{3N}}^{m(h_n^{3N}+1)-1}a_i-h_n-t_N \ge 0
\end{align*}
we have $q^{(j)} \ge mh_n^{3N}$ for $j=1,\cdots,h_n$. By (\ref{generic_stacks_lemma}), we have
\begin{multline*}
    \PP\left(\left\{a \in \Omega:\max\{\bar{X}^{(r,1)}_{q^{(1)}}(v), \cdots,\bar{X}^{(r,h_n)}_{q^{(h_n)}}(v)\} \ge 2h_r^{-(N-N_n^{(0)})}\right\}\right) \\ \le (h_n+1)\exp(-2h_r^N).
\end{multline*}

The probability that more than $2h_r^{-(N_n-N_n^{(0)})} \cdot |K_n|$ non-spacer levels share the same $\phi_{\xi_r,A,N_n}$-word is thus below
\[P(N):=(h_n+1)\sum_{r=1}^n(h_r+1)^N\exp(-2h_r^N)\]
which converges to 0 as $N \to \infty$. Therefore, by picking $N_n > N_n^{(0)}$ such that $P(N_n)<1$ and $q_n=m(h_n^{3N_n}+1)$, there exists $a_{n,1},\cdots,a_{n,q_n-1} \in \{0,1\}$ such that for $r=1,\cdots,n$, every $v \in \Sigma_r^{N_n}$ is the $\phi_{\xi_r,A,N_n}$-word for no more than $2h_r^{-(N_n-N_n^{(0)})} \cdot |K_n|$ non-spacer levels in $S_{n+1}$.
\end{proof}

\section{Blow-up of sequence entropy} \label{sec:seq_ent_blowup}

We now begin to construct a rank one system that satisfies the desired property. Intuitively, suppose we arrive at the $n$-th tower $S_n$ in the cut-and-stack procedure described in the remark of Definition \ref{rank_one_system}. We choose $N_n$ large enough so that the applicability of Proposition \ref{Bernoulli_pushforward} holds for a large portion of $t_1,\cdots,t_{N_n}$ that tends to 1 as $n \to \infty$. For this $N_n$, we hope to achieve the following:
\begin{itemize}
    \item[$(*)$] Choose the number of divisions $q_n \in \NN$ and a sufficiently irregular binary sequence $a_{n,1},\cdots,a_{n,q_n-1} \in \{0,1\}$ such that, for $r=1,\cdots,n$, every possible distinguishable orbit relative to the partitions induced by $S_r$ covers a measure below the order of $h_r^{-N_n}$.
\end{itemize}
 By Proposition \ref{disjoint_lower_bound}, this ensures that the size of sequence entropy is bounded from below by an expression on the order of $N_n \log h_r$. Proposition \ref{generic_stacks} guarantees the existence (in fact genericness) of parameters that satisfy $(*)$ for sufficiently large $q_n$.

\begin{proof}[Proof of Theorem \ref{seq_ent_blowup}]
    Our construction of the desired rank one system $(X,\cB,\mu,T)$ begins with the initial tower $S_1= \bigcup_{k=0}^{h_1-1}T^kF_1$ where $\mu(F_1)$ can ultimately be chosen to satisfy $\mu(X)=1$. Suppose that $S_i=\bigsqcup_{k=0}^{h_i-1}T^kF_i$ where $i = 1,\cdots,n$ is obtained after $n$ steps of the cut-and-stack procedure.

    \medskip
    
    By Proposition \ref{generic_stacks}, we may pick $N_n^{(0)}, N_n \in \NN$ with $N_n \ge nN_n^{(0)}$ sufficiently large, $q_n \in \NN$ dependent on $N_n$, and $a_1,\cdots,a_{q_n-1} \in \{0,1\}$ such that the resulting tower $S_{n+1}$ satisfies that, for $1 \le r \le n$, every $v \in \Sigma_r^{N_n}$ can be expressed in the form $\phi_{\xi_r,A,N_n}(T^kF_{n+1})$ for no more than $2h_r^{-(N_n-N_n^{(0)})} \cdot |K_n|$ distinct values of $k \in K_n$, where
    \[|K_n| \ge h_{n+1}-t_{N_n}-a_n \approx h_{n+1}\]
    from our discussion in Subsection~\ref{ssec:generic_stacks_prep}. Let \[X_{n+1}=\bigsqcup_{k \in K_n}T^kF_{n+1},\] then $\mu(X_{n+1}) = |K_n| \cdot \mu(F_{n+1}) \to 1$ and for any $v \in \Sigma_r^{N_n}$, we have
    \begin{align*}
        \mu(\phi_{\xi_r,A,N_n}^{-1}(v) \cap X_{n+1}) & \le 2h_r^{-(N_n-N_n^{(0)})} \cdot |K_n| \cdot \mu(F_{n+1}). \\
        & \le 2h_r^{-(N_n-N_n^{(0)})}.
    \end{align*}
    Since $f(x)$ is increasing for small $x$, applying Proposition \ref{disjoint_lower_bound} gives
    \begin{align*}
        \frac{1}{N_n}H\left(\bigvee_{k=1}^{N_n} T^{-t_k}\xi_r\right) & \ge \frac{1}{N_n}\sum_{w \in \Sigma_r^{N_n}}f(\mu(\phi_{\xi_r,A,N_n}^{-1}(w) \cap X_{n+1})) \\
        & \ge -\frac{1}{N_n} \cdot \mu(X_{n+1}) \cdot \log \left(2h_r^{-(N_n-N_n^{(0)})}\right)
    \end{align*}
    and hence \[h^A(T,\xi_r) = \varlimsup_{n \to \infty}\frac{1}{n}H\left(\bigvee_{k=1}^{n} T^{-t_k}\xi_r\right) \ge \log h_r\]
    for all $r$. Therefore,
    \[h^A(T) = \sup_\xi h^A(T,\xi) \ge \varlimsup_{r \to \infty} h^A(T,\xi_r) = +\infty.\]
\end{proof}

\begin{proof}[Proof of Corollary~\ref{seq_ent_blowup.2}]
    Denote by $\underline{d}(J)$ the lower density of $J$, then for any measurable partition $\xi$ we have
    \begin{align*}
        h^A(T,\xi) & \ge\varlimsup_{n \to \infty} \frac{1}{n}H\left(\bigvee_{k \in J \cap [1,n]} T^{-t_k}\xi\right) \\
        & \ge \varliminf_{n \to \infty} \frac{|J \cap [1,n]|}{n} \cdot \varlimsup_{n \to \infty} \frac{1}{|J \cap [1,n]|}H\left(\bigvee_{k \in J \cap [1,n]} T^{-t_k}\xi\right) \\
        & = \underline{d}(J) \cdot h^{A_J}(T,\xi).
    \end{align*}
    The result follows from applying Theorem~\ref{seq_ent_blowup} to $A_J$.
\end{proof}

\section{Cases of zero sequence entropy} \label{sec:zero_seq_ent}

This section will provide a clear picture of the point of transition from zero to positive sequence entropy. To obtain a positive sequence entropy, a rank one system must be built through a sufficiently complex cut-and-stack procedure, and this puts a minimal requirement on the growth rate of tower heights. 

\subsection{Preliminaries}

We will use the following generalisation of a well-known fact about conditional entropy which was already proven in \cite{Kushnirenko67}:
\begin{prop} \label{close_partitions_bound}
Let $\xi,\eta$ be two measurable partitions of $X$, then
\[\frac{1}{n}H\left(\bigvee_{k=1}^n T^{-t_k}\eta\right) \le \frac{1}{n}H\left(\bigvee_{k=1}^n T^{-t_k}\xi\right)+H(\eta|\xi).\]
\end{prop}

This inequality provides a useful method to compute an upper bound for sequence entropy:
\begin{prop} \label{slow_upper_bound}
    Let $\{\xi_n\}$ be a refining sequence of partitions of $X$ which approach the point partition, then
    \[h^A(T) \le \inf_\tau \left\{ \varlimsup_{n \to \infty} \frac{1}{n}H\left(\bigvee_{k=1}^n T^{-t_k}\xi_{\tau(n)}\right)\right\},\]
    where the infimum is taken over all non-decreasing functions $\tau:\NN \to \NN$ that diverge to infinity.
\end{prop}

\begin{proof}
    Let $\eta$ be any measurable partition of $X$, then by Proposition~\ref{close_partitions_bound} we have
    \[\frac{1}{n}H\left(\bigvee_{k=1}^n T^{-t_k}\eta\right) \le \frac{1}{n}H\left(\bigvee_{k=1}^n T^{-t_k}\xi_{\tau(n)}\right)+H(\eta|\xi_{\tau(n)})\]
    for any $\tau$. If $\tau(n) \to \infty$, then $H(\eta | \xi_{\tau(n)}) \to 0$ and hence
    \[h^A(T,\eta) \le \varlimsup_{n \to \infty} \frac{1}{n}H\left(\bigvee_{k=1}^n T^{-t_k}\xi_{\tau(n)}\right)\]
    for any measurable partition $\eta$.
\end{proof}

A direct application of Jensen's inequality yields the following:
\begin{prop} \label{disjoint_upper_bound}
Let $E = \bigsqcup_{i=1}^mE_i$ be a disjoint union of measurable subsets of $X$, then \[\sum_{i=1}^m -\mu(E_i) \log\mu(E_i) \le -\mu(X) \log \mu(X) + \mu(X) \log m.\]
\end{prop}

\begin{rem}
    Proposition~\ref{disjoint_upper_bound} allows us to obtain a crude upper bound that loosely represents the topological sequence entropy of a symbolic dynamical system; this is sufficient for some purposes such as in the case of Theorem~\ref{zero_seq_ent}.
\end{rem}

\subsection{Outline and proof}

The proof of Theorem~\ref{zero_seq_ent} embodies a mechanism with an interesting narrative. The story begins with Proposition~\ref{slow_upper_bound} that comes with a general principle to control the sequence entropy of any measure-preserving system $X$: better upper bounds are obtained the slower we take increments of the referencing partition $\xi_{\tau(N)}$. Once this partition is fixed for any given sampling length $N$, a quantitative upper bound based on the number of $\phi_{\xi_{\tau(N)},A,N}$-words in $X$ can then be obtained by Proposition~\ref{disjoint_upper_bound}.

\medskip

In the case of a rank one system where we take the induced partitions $\xi_n=\xi_{S_n}$, the linear structure of these systems allows us to obtain an upper bound for this word count through a slick combinatorial argument. However, this combinatorial count becomes ineffective when the reference partition corresponds to an approximating tower $S_n$ that is far too short compared to the $A$-orbit.

\medskip

The end product of the mechanism is a balancing condition for $\tau$ that captures this fierce struggle between the slowness of $\tau$ demanded by Proposition~\ref{slow_upper_bound} (Equation (\ref{1.2.eq6})) and the gradual deterioration of the upper bound in Proposition~\ref{disjoint_upper_bound} from the combinatorial count of words (Equation (\ref{1.2.eq4})).

\begin{rem}
    Chronologically, a weaker version of Theorem~\ref{zero_seq_ent}~\ref{zero_seq_ent.1} was obtained before Theorem \ref{seq_ent_blowup} and has served as a valuable guide which highlights the necessity of rapid growth of tower heights. Subsequent results in Theorem~\ref{zero_seq_ent}  were later obtained by varying over diverging functions $\tau$ as enabled by Proposition~\ref{slow_upper_bound}.
\end{rem}

\begin{proof}[Proof of Theorem~\ref{zero_seq_ent}]

Let $A=\{t_n\}$ be a subexponential but superlinear sequence and denote its maximal gap by $s_n = \max_{1 \le i \le n-1}\{t_{i+1}-t_i\}$ for $n \ge 2$, then $s_n \to \infty \ (n \to \infty).$ Let $\{c_n\}$ be an arbitrary sequence of positive integers and define
\[\tau(N) = \min\set{n \in \NN}{h_n>\frac{t_N}{c_N}}.\]

Let $\cV$ be the set of all $\phi_{\xi_n,A,N}$-words in $X$ where $n=\tau(N)$. Since $t_N<c_Nh_n$, any $w \in \cV$ is of the form $v_00^{b_1}v_10^{b_2} \cdots v_{c_N-1}0^{b_{c_N}}v_{c_N}$ where
\begin{enumerate}
    \item $0 \le b_i \le N$ and $b_1+\cdots+b_{c_N} \le N$;
    \item $v_0$ is empty or begins with one of the symbols $1,\cdots,h_n$;
    \item $v_1,\cdots,v_{c_N}$ are empty or begin with one of the symbols $1,\cdots,s_N$;
\end{enumerate}
hence \[|\cV| \le \binom{N+c_N}{N}(h_n+1)(s_N+1)^{c_N}.\]
By Proposition~\ref{disjoint_upper_bound}, we have
\begin{equation*}
    H\left(\bigvee_{k=1}^N T^{-t_k}\xi_{\tau(N)}\right) = \sum_{w \in \cV} f(\mu\phi_{\xi_n,A,N}^{-1}(w)) \le \log |\cV|
\end{equation*}
and thus by Proposition~\ref{close_partitions_bound},
\begin{multline} \label{1.2.eq3}
    \frac{1}{N}H\left(\bigvee_{k=1}^N T^{-t_k}\eta\right) \le \frac{1}{N} \log (h_{\tau(N)}+1) \\ + \frac{c_N}{N} \log (s_N+1) + \frac{1}{N}\log\binom{N+c_N}{N} + H(\eta|\xi_{\tau(N)}).
\end{multline}

Suppose that $\{c_n\}$ satisfies
\begin{equation} \label{1.2.eq4}
    \lim_{n \to \infty}\frac{c_n}{n} \log s_n = 0,
\end{equation}
then $\lim_{n \to \infty}\frac{c_n}{n}=0$. From here it can be shown that
\begin{equation} \label{1.2.eq5}
    \lim_{n \to \infty} \frac{1}{n} \log \binom{n+c_n}{n}=0,
\end{equation}
and since $\frac{t_n}{c_n}\to \infty \ (n \to \infty)$ we have $\tau(N) \to \infty \ (N \to \infty)$. On the other hand, if $\vphi:\RR^+ \to \RR^+$ is an increasing function such that \[\varlimsup_{n \to \infty}\frac{\log h_{n+1}}{\vphi(h_n)} < \infty,\] then there exists $M>0$ such that
\[\log h_{\tau(N)} < M\vphi\left(h_{\tau(N)-1}\right) \le M\vphi\left(\frac{t_N}{c_N}\right).\]
We can therefore conclude from (\ref{1.2.eq3}) that $h^A(T,\eta)=0$ for all $\eta$ if $\{c_n\}$ further satisfies
\begin{equation} \label{1.2.eq6}
    \lim_{n \to \infty}\frac{1}{n} \vphi\left(\frac{t_n}{c_n}\right) = 0.
\end{equation}

The last part of the proof is to balance the conditions (\ref{1.2.eq4}) and (\ref{1.2.eq6}):

\begin{enumerate}
    \item If $\vphi(n)=\log n$, the sequence $c_n=1$ satisfies both conditions as long as $\{t_n\}$ is subexponential.
    \item If $\vphi(n)=n^\beta$ and $t_n = O(n^\alpha)$, we have
    \begin{align*}
        & \text{there exists } \{c_n\} \subseteq \NN \text{ such that} \begin{cases}
            \displaystyle \lim_{n \to \infty}\frac{c_n}{n} \log s_n = 0, \\
            \displaystyle\lim_{n \to \infty}\frac{1}{n} \vphi\left(\frac{t_n}{c_n}\right) = 0
        \end{cases} \\
        \iff & \text{there exists } \{c_n\} \subseteq \NN \text{ such that} \begin{cases}
            \displaystyle c_n \ll n/\log n \\
            \displaystyle c_n \gg n^{\alpha-1/\beta}
        \end{cases} \quad (n \to \infty) \\
        \iff & n^{\alpha-1/\beta} \ll n/\log n \quad (n \to \infty) \\
        \iff & \alpha<1+\frac{1}{\beta}
    \end{align*}
    where we take $c_n=\lfloor n^{(\alpha-1/\beta+1)/2}\rfloor$.
    \item If $\vphi(n)=e^{n^\beta}$ and $t_n = \lfloor Cn(\log n)^\alpha \rfloor$ where $\alpha\beta < 1$, the sequence $c_n=\lfloor n/(\log\log n)^2\rfloor$ satisfies
    \[\varlimsup_{n \to \infty} \frac{c_n}{n} \log s_n \le \varlimsup_{n \to \infty} \frac{c_n(\alpha\log\log n+O(1))}{n} = 0\]
    and
    \[\varlimsup_{n \to \infty} \frac{1}{n} \vphi\left(\frac{t_n}{c_n}\right) = \varlimsup_{n \to \infty} n^{(C(\log\log n)^2)^\beta(\log n)^{\alpha\beta-1}-1} = 0.\]
\end{enumerate}
\end{proof}

\begin{rem} On the proof of Theorem~\ref{zero_seq_ent}:
\begin{itemize}
    \item The balancing conditions in (\ref{1.2.eq4}) and (\ref{1.2.eq6}) provides a general framework to calculate a critical threshold of tower growth rate for any specific choice of $A=\{t_n\}$. The mechanism breaks down if $\{t_n\}$ (and therefore $\{s_n\}$) is exponential or above, as (\ref{1.2.eq4}) can no longer hold regardless of the choice of $\{c_n\}$.
    \item Since (\ref{1.2.eq5}) is logically deduced from (\ref{1.2.eq4}), the effect of the number of spacer levels $a_{n,i}$ added between slices of $S_n$, in itself, is dominated by its own effect on perturbing the pre-existing levels in $S_n$ and does not contribute essentially to the sequence entropy. This suggests that the assumption $a_{n,i} \in \{0,1\}$ throughout the proof of Theorem \ref{seq_ent_blowup} does not detract too far from the general theory.
    \item A better upper bound estimate of $|\cV|$ is in fact available, however, if a uniform upper bound for $a_{n,i}$ exists, as this inhibits the ability of spacer levels in perturbing the level sets in $S_n$. We will take advantage of this observation in Section~\ref{sec:weak_seq_ent_flex}.
\end{itemize}
\end{rem}

\subsection{Establishing the critical threshold}

The proof of Theorem~\ref{pos_seq_ent} utilises a simplified version of the proof of Theorem~\ref{seq_ent_blowup}.

\begin{proof}[Sketch of proof of Theorem~\ref{pos_seq_ent}] Note that one only needs to demonstrate that $h^A(T,\xi)>0$ for the initial tower $\xi = \xi_{S_1}$ to which we may set $h_1=2$. Define
\[\iota(a) = \Phi(S_n)0^{a_1}\Phi(S_n)0^{a_2}\cdots \in \Sigma^\NN\]
where
\[\Phi(S_n) = \textrm{AB}0^{b_{n,1}}\textrm{AB}0^{b_{n,2}} \cdots \textrm{AB}0^{b_{n,g_n-1}}\textrm{AB} \quad (b_{n,i} \in \{0,1\}).\]
(The coding symbols in $\Sigma=\{0,\textrm{A},\textrm{B}\}$ were renamed for added clarity.) Then for any $k \in \NN$ and $l_0,l_1 \in \Sigma$ with $\PP(\iota(a)_k = l_0) > 0$,
\[\PP(\iota(a)_{k+s}=l_1 \mid \iota(a)_{k}=l_0) \le \frac{1}{2}\]
as long as $s \ge h_n+1$. An analogous result to Proposition~\ref{generic_stacks} is then obtained by setting the following parameters:
\begin{itemize}
    \item For case \ref{pos_seq_ent.1}, set $N_n^{(0)} = \lfloor h_n^{\beta} \rfloor + 1$ so that for $k \ge N_n^{(0)}$,
    \[t_{k+1}-t_k \ge \alpha h_n^{(\alpha-1)\beta}-1 > h_n+1\]
    for sufficiently large $n$ as $\alpha \ge 1+1/\beta$, and
    \[\varlimsup_{n \to \infty} \frac{\log 2^{N_n^{(0)}}}{h_n^{\beta}} = \log 2 < \infty.\]
    \item For case \ref{pos_seq_ent.2}, set $N_n^{(0)} = \lfloor e^{h_n^\beta+1} \rfloor + 1$ so that for $k \ge N_n^{(0)}$,
    \[t_{k+1}-t_k \ge e^\alpha h_n^{\alpha\beta} > h_n+1\]
    for sufficiently large $n$ as $\alpha \ge 1/\beta$, and
    \[\varlimsup_{n \to \infty} \frac{\log 2^{N_n^{(0)}}}{e^{h_n^{\beta}}} = e\log 2 < \infty.\]
    \item In any case, set $q_n=2^{N_n^{(0)}}$ and $N_n=\gamma N_n^{(0)}$ for some $\gamma>1$ so that
    \[\PP(\iota(a)_{k+t_1} \cdots \iota(a)_{k+t_{N_n}} = l_1 \cdots l_{N_n}) \le 2^{-(1-\frac{1}{\gamma})N_n},\]
    and for $\veps > 0$ set $\delta = 2^{-(\frac{1}{2\gamma}-\veps)N_n}$ so that for sufficiently large $n$, $q_n\delta^2 = 2^{2\veps N_n}$ forces eventual vanishing of the total probability of failing sequences
    \[P(N_n) = 3^{N_n}\exp(-2\lfloor q_n/t_{N_n} \rfloor \delta^2).\]
\end{itemize}
By inheriting notations from Section~\ref{sec:seq_ent_blowup}, this establishes the existence of parameters $a_1,\cdots,a_{q_n-1}$ that satisfies
\[\mu(\phi_{\xi,A,N_n}^{-1}(v) \cap X_{n+1}) \le 2^{-cN_n+1}\]
for all $v \in \Sigma^{N_n}$, where \[c=\max\left\{1-\frac{1}{\gamma},\frac{1}{2\gamma}-\veps\right\}.\]
Thus we obtain, by setting $\gamma=3/2$ and taking $\veps \to 0$,
\[h^A(T,\xi) \ge \frac{1}{3} \log 2 > 0.\]
\end{proof}

\section{Flexibility of sequence entropy} \label{sec:weak_seq_ent_flex}

We now attempt to tackle the flexibility problem for sequence entropy along polynomial sequences of the form $t_n=\lfloor n^\alpha \rfloor$ by combining the techniques developed thus far. We allow the number of spacer levels added between each slice to range over $0,1,\cdots,L-1$ for some integer $L \ge 2$. The resulting system remains to be of rank one as the total number of spacers $a_n = \sum_{i=1}^{q_n-1}a_{n,i}$ added onto $S_n$ satisfies
\[\sum_{n=1}^\infty \frac{a_n}{h_{n+1}} \le \sum_{n=1}^\infty \frac{q_n(L-1)}{q_nh_n} < +\infty.\]

The sequence of tower heights used in the following construction adheres to the critical threshold obtained in Theorem~\ref{pos_seq_ent}~\ref{pos_seq_ent.1}.

\begin{proof}[Proof of Theorem~\ref{weak_flex_seq_ent}]
    Let $L \ge 2$ be a fixed integer. The construction of the rank one system with spacer levels bounded by $L-1$ begins with the initial tower $S:=S_1 = \bigsqcup_{k=0}^{h_1-1}T^kF_1$ where $h:=h_1>L$. Let $\xi:=\xi_S$ be the partition induced by the initial tower, and denote by $\xi' \succ \xi$ the natural refinement by the Kakutani skyscraper on $F_1$ as described in Subsection~\ref{ssec:refinement_induced}, to be determined recursively alongside the construction of $X$. Since $a_{n,i}$ is uniformly bounded by $L-1$, $\xi'$ will be a finite partition with $h+L-1$ atoms. Under the limiting process, the coding function $\phi_{\xi'}:X \to \Sigma'$ will be defined a.e. on $X$ with $\Sigma'=\{1,\cdots,h+L-1\}$. The coding maps $\Phi$ and $\Phi'$ will refer to the $\phi_\xi$- and $\phi_{\xi'}$-names of a tower respectively.
    
    Suppose that the towers $S_1,\cdots,S_n$ have already been determined, so that $S_n$ is composed by $g_n=\prod_{k=1}^{n-1}q_k$ slices of $S_1$ with sequence of spacer counts $b_{n,i} \in \{0,1,\cdots,L-1\}$ where $1 \le i \le g_n-1$. The $\phi_{\xi'}$-word of $S_n$ is well-defined and is given by
    \[\Phi'(S_n) = \Phi(S)c_{n,1}\Phi(S)c_{n,2} \cdots \Phi(S)c_{n,g_n-1}\Phi(S)\]
    where
    \[c_{n,i} = (h+1,\cdots,h+b_{n,i}).\]

    Let $(\Omega,\cF,\PP,\sigma) \sim \mathscr{B}(\frac{1}{L},\cdots,\frac{1}{L})$ be the Bernoulli shift on the alphabet $\{0,1,\cdots,L-1\}$, and define $\iota_n:\Omega \to (\Sigma')^\NN$ by
    \[\iota_n(a) = \Phi'(S_n)c(a_1)\Phi'(S_n)c(a_2) \cdots\]
    where for $0 \le l \le L-1$ denote
    \[c(l) = (h+1,\cdots,h+l).\]

    Let
    \[N_n = \left\lfloor\left(\frac{h_n-2}{\alpha}\right)^\beta\right\rfloor-1,\]
    then for $k \le N_n$,
    \begin{equation} \label{eq:weak_flex.1}
        t_{k+1}-t_k \le (k+1)^\alpha-k^\alpha+1 \le \alpha(k+1)^{\alpha-1}+1 \le h_n-1.
    \end{equation}

    For each $v \in (\Sigma')^{N_n}, j\in [1,h_n] \cap \ZZ$, define a sequence of random variables $(X_k^{(j)}(v))_{k=1}^\infty$ by
    \[X_k^{(j)}(v) = \begin{cases}
        1 & (\text{if } \iota_n(\sigma^{k-1}(a))_{j+t_1} \cdots \iota_n(\sigma^{k-1}(a)) = v); \\ 
        0 & (\text{otherwise}),
    \end{cases}\]
    then (\ref{eq:weak_flex.1}) guarantees that there is at most one possible choice of the collection of spacer counts $c_i$ which successfully represents $v$, hence
    \[\PP(\iota_n(\sigma^{k-1}(a))_{j+t_1} \cdots \iota_n(\sigma^{k-1}(a))_{j+t_{N_n}} = v) \le L^{-\lfloor (t_{N_n}-1)/(h_n+L)\rfloor}.\]
    Fix $\veps > 0$, then since
    \[\left\lfloor \frac{t_{N_n}-1}{h_n+L} \right\rfloor \approx \frac{1}{\alpha}N_n \approx \frac{1}{\alpha^{1+\beta}}h_n^\beta\]
    we have
    \[\EE X_k^{(j)}(v) \le L^{-(\alpha^{-1-\beta} -\veps)h_n^\beta}\]
    for sufficiently large $n$.

    Let $m = \lfloor t_{N_n}/h_n \rfloor+2$, then $h_n m > h_n+t_n$ and hence $(X_k^{(j)}(v))_{k=1}^\infty$ is $m$-dependent, thus for all $q \in \NN$ and $\delta>0$,
    \[\PP(\bar{X}_q^{(j)}(v)-\EE\bar{X}_q^{(j)}(v) \ge \delta) \le \exp\left(-2\lfloor q/m \rfloor\delta^2\right)\]
    where $\bar{X}_q^{(j)}(v) = \frac{1}{q}\sum_{k=1}^qX_k^{(j)}(v)$ satisfies
    \[\EE\bar{X}_q^{(j)}(v) \le L^{-(\alpha^{-1-\beta} -\veps)h_n^\beta}.\]
    Pick $q=q_n=\lfloor\kappa^{h_n^\beta}\rfloor+1, \delta=\kappa^{-(\frac{1}{2}-\veps)h_n^\beta}$ for some $\kappa > 1$, then $q_n\delta^2 \ge \kappa^{2\veps h_n^\beta}$ and
    \begin{align*}
        P(N_n) & := \sum_{j=1}^{h_n}\sum_{v \in (\Sigma')^{N_n}} \PP(\bar{X}_q^{(j)}(v) \ge L^{-(\alpha^{-1-\beta} -\veps)h_n^\beta}+\kappa^{-(\frac{1}{2}-\veps)h_n^\beta}) \\
        & \le (h+L)^{N_n}\exp(-2\lfloor q_n/m \rfloor \delta^2)
    \end{align*}
    converges to 0 as $n \to \infty$. Thus, for large $n$, we can pick $a_{n,1},\cdots,a_{n,q_n-1} \in \{0,1,\cdots,L-1\}$ such that if $S_{n+1}=\bigsqcup_{k=0}^{h_{n+1}-1}T^kF_{n+1}$ is defined by
    \[\Phi(S_{n+1})=\Phi(S_n)0^{a_{n,1}}\Phi(S_n)0^{a_{n,2}} \cdots \Phi(S_n)0^{a_{n,q_n-1}}\Phi(S_n),\]
    then
    \[\Phi'(S_{n+1})=\Phi'(S_n)c(a_{n,1})\Phi'(S_n)c(a_{n,2}) \cdots \Phi'(S_n)c(a_{n,q_n-1})\Phi'(S_n)\]
    is well-defined and each $v \in (\Sigma')^{N_n}$ is represented as $\phi_{\xi',A,N_n}(T^kF_{n+1})$ for no more than
    \[|K_n| \cdot 2\max\{L^{-(\alpha^{-1-\beta} -\veps)},\kappa^{-(\frac{1}{2}-\veps)}\}^{h_n^\beta}\]
    distinct choices of $k \in K_n$ where
    \[K_n = \set{k \in \ZZ}{0 \le k \le h_{n+1}-t_{N_n}-1,T^kF_{n+1} \subseteq S_n}.\]
    Let $X_{n+1} = \bigsqcup_{k \in K_n}T^kF_{n+1} \subseteq S_{n+1}$, then
    \[\mu(\phi^{-1}_{\xi',A,N_n}(v) \cap X_{n+1}) \le 2\max\{L^{-(\alpha^{-1-\beta} -\veps)},\kappa^{-(\frac{1}{2}-\veps)}\}^{h_n^\beta}\]
    for all $v \in (\Sigma')^{N_n}$, and we have in general
    \[\mu(X_{n+1}) \ge \frac{h_{n+1}-t_{N_n}-a_n}{h_{n+1}}\mu(S_{n+1}) \to 1 \quad (n \to \infty),\]
    as established in Subsection~\ref{ssec:generic_stacks_prep}. Therefore
    \begin{align*}
        h^A(T,\xi') & \ge \varlimsup_{n \to \infty} \frac{1}{N_n}\sum_{v \in (\Sigma')^{N_n}} f(\mu(\phi^{-1}_{\xi',A,N_n}(v) \cap X_{n+1})) \\
        & \ge \varlimsup_{n \to \infty} \frac{h_n^{\beta}}{N_n} \cdot \mu(X_{n+1}) \cdot \left(- \log \max\left\{L^{-(\alpha^{-1-\beta} -\veps)},\kappa^{-(\frac{1}{2}-\veps)}\right\}\right) \\
        & = \alpha^\beta \log \min\{L^{\alpha^{-1-\beta} -\veps},\kappa^{\frac{1}{2}-\veps}\}.
    \end{align*}
    Taking $\veps \to 0$ thus gives
    \[h^A(T) \ge \alpha^\beta \log \min\{L^{\alpha^{-1-\beta} },\kappa^{\frac{1}{2}}\}.\]

    On the other hand, for $q_n=\lfloor \kappa^{h_n^\beta} \rfloor + 1$ we have
    \[\lim_{n \to \infty} \frac{\log h_{n+1}}{h_n^\beta} = \kappa.\]
    Define
    \[\tau(N) = \min\set{n \in \NN}{h_n \ge \frac{t_N}{\lambda N}}\]
    for some $\lambda>0$, then $\tau(N) \to \infty$ and each $\phi_{\xi_{\tau(N)},A,N}$-word spans over at most $\lceil \lambda N \rceil+1$ copies of $\Phi_{\tau(N)}(S_{\tau(N)})$ and hence is determined by the first letter and at most $\lceil \lambda N \rceil$ choices of spacer counts. Therefore, the number of attainable $\phi_{\xi_{\tau(N)},A,N}$-words is below $(h_{\tau(N)}+1)L^{\lambda N+1}$. By Propositions~\ref{slow_upper_bound} and \ref{disjoint_upper_bound},
    \begin{align*}
        h^A(T) & \le \varlimsup_{N \to \infty}\frac{1}{N}\left(\log(h_{\tau(N)}+1)+(\lambda N+1)\log L\right) \\
        & \le \varlimsup_{N \to \infty}\left(\lambda^{-\beta}\log\kappa+\lambda\log L\right) \\
        & = (1+\beta) \cdot \left(\frac{1}{\beta^\beta}\log \kappa \cdot (\log L)^\beta\right)^{\frac{1}{1+\beta}}
    \end{align*}
    where we pick $\lambda$ such that $\lambda^{-\beta} \log \kappa = \frac{1}{\beta}\lambda\log L$.
    
    We conclude that the resulting rank one system $(X,\cB,\mu,T)$ satisfies
    \[\alpha^\beta \log \min\{L^{\alpha^{-1-\beta} },\kappa^{\frac{1}{2}}\} \le h^A(T) \le (1+\beta) \cdot \left(\frac{1}{\beta^\beta}\log \kappa \cdot (\log L)^\beta\right)^{\frac{1}{1+\beta}}.\]
    Part~\ref{small_flex} is obtained by picking $L=2$ and $\kappa=1+\veps$ for sufficiently small $\veps$. For part~\ref{large_flex} we pick $\kappa=L^{2\alpha^{-1-\beta}}$.
\end{proof}

\begin{rem}
    To obtain stronger flexibility results, a better estimate than Proposition~\ref{disjoint_upper_bound} of an upper bound that tailors specifically to metric sequence entropy may be in order as the variational principle does not apply to sequence entropy. Unfortunately, the result of Theorem~\ref{weak_flex_seq_ent} does not generalise immediately to abstract subexponential sequences as the method used in Section~\ref{sec:zero_seq_ent} to obtain an upper bound for $h^A(T)$ relies on concrete specifications on the sequence $A$ itself.
\end{rem}

\bibliography{refs}
\end{document}